\documentclass{amsart}
\usepackage{amsmath,amscd,amsfonts,amssymb}
\newtheorem{thm}{Theorem}[section]
\newtheorem{cor}[thm]{Corollary}
\newtheorem{lem}[thm]{Lemma}

\theoremstyle{definition}

\theoremstyle{remark}

\numberwithin{equation}{section}

\def\.{{\cdot}}

\def\<{\langle} \def\>{\rangle}

\def\lead{\leaders\hbox to 1.5ex{\hss${.}$\hss}\hfill}
\def\arr{\hbox to 40pt{\rightarrowfill}}
\def\larr{\hbox to 40pt{\leftarrowfill}}

\long\def\alert#1{\parindent2em\smallskip\hbox to\hsize%
{\hskip\parindent\vrule%
\vbox{\advance\hsize-2\parindent\hrule\smallskip\parindent.4\parindent%
\narrower\noindent#1\smallskip\hrule}\vrule\hfill}\smallskip\parindent0pt}

\begin{document}

\title[Estimations of the low dimensional homology of ...]
 {Estimations  of the low dimensional homology of  Lie algebras with large abelian ideals}

\author[P. Niroomand]{Peyman Niroomand}
\address{Department of Pure Mathematics\\
Damghan University, Damghan, Iran}
\email{p$\_$niroomand@yahoo.com}

\author[F.G.  Russo]{Francesco G. Russo}
\address{DIEETCAM\\
Universit\'a degli Studi di Palermo\\
Viale Delle Scienze, Edificio 8, 90128, Palermo, Italy}
\email{francescog.russo@yahoo.com}



\date{\today}

\begin{abstract}
A  Lie algebra $L$ of dimension $n \ge1 $ may be classified, looking for restrictions of the size on its second integral homology Lie algebra $H_2(L,\mathbb{Z})$,  denoted by $M(L)$ and often called Schur multiplier of $L$. In case $L$ is nilpotent,  we proved that $\mathrm{dim} \ M(L) \leq \frac{1}{2}(n+m-2)(n-m-1)+1$,  where $\mathrm{dim} \ L^2=m \ge 1$, and worked on this bound under various perspectives. In the present paper, we estimate the previous bound for $\mathrm{dim} \ M(L) $ with respect to other inequalities of the same nature. Finally, we provide  new upper bounds for the Schur multipliers of pairs and triples of nilpotent Lie algebras, by means of  certain exact sequences due to Ganea and Stallings in their original form.
\end{abstract}

\maketitle
\section{Previous contributions and statement of the results}
The classification of finite dimensional nilpotent Lie algebras has interested the works of several authors both in topology and in algebra, as we can note from \cite{anc1, anc3}.
The second integral homology Lie algebra $H_2(L,\mathbb{Z})$ of a nilpotent Lie algebra $L$ of dimension $\mathrm{dim} \ L =n$ is again a finite dimensional Lie algebra and its dimension may be connected with that of $L$ under many points of view. It is customary to call $H_2(L,\mathbb{Z})$ the \textit{Schur multiplier} of $L$ and denote it with $M(L)$, in analogy with the case of groups due to Schur (see \cite{karp,rot}).  The study of the numerical conditions among  $\mathrm{dim} \ L$ and  $\mathrm{dim} \ M(L)$ is the subject of many investigations.

A classic contribution of Batten and others \cite{es1} shows that \begin{equation}\label{batten}\mathrm{dim} \ M(L) \leq \frac{1}{2} n (n-1).\end{equation}
Denoting  $\mathrm{dim} \ L^2=m \ge 1$,
Yankosky \cite{yank} sharpened \eqref{batten} by
\begin{equation}\label{yank} \mathrm{dim} \ M (L) \le \frac{1}{2} (m - m^2 + 2mn - 2n)\end{equation}
in which the role of $m$ is significant. We  contributed in \cite{n1, n2, n3, n4, n6} under various aspects and showed that  \begin{equation}\label{nice}\mathrm{dim} \ M(L) \leq \frac{1}{2}(n+m-2)(n-m-1)+1,\end{equation} is better than \eqref{batten} and \eqref{yank}. The crucial step was to prove that \eqref{nice} is better than \eqref{yank} unless  small values of $m$  and $n$ occur (see \cite[Corollary 3.4]{n2}) and this happens  most of the times.

Another  inequality of the same nature of \eqref{yank} and \eqref{nice}  is given by
\begin{equation}\label{salemkar}\mathrm{dim} \ M(L) \leq \mathrm{dim} \ M(L/L^2) + \mathrm{dim} \ L^2 (\mathrm{dim} \ L/Z(L) -1).\end{equation}
and can be found in \cite[Corollary 3.3]{sal1}, where it is  shown that
it is  better than  \eqref{batten}.

We will concentrate on \eqref{batten}, \eqref{nice} and \eqref{salemkar} in the present paper, but inform the reader that there are other inequalities  in \cite{sal1, bosko1, bosko2, sal2, sal3, n1, n2, n3, n4, n6} which may be treated in a similar way, once  an appropriate study is done.

Most of the above bounds is in fact inspired by analogies with the case of groups (see  \cite{es1,es2,n5}), even if careful distinctions should be made between the two contexts. Here we  do qualitative considerations among \eqref{yank}, \eqref{nice} and \eqref{salemkar}, because it is not easy to compare   all of them and  a specific study is necessary.

On the other hand, the importance of \eqref{yank}, \eqref{nice} and \eqref{salemkar} is due to the invariants which appear. In order to understand this point, we should recall that the idea of classifying nilpotent Lie algebras of finite dimension by restrictions on  their Schur multipliers goes back to  \cite{es1, es2} and continued in  \cite{sal1, bosko1, bosko2, sal2, sal3} under different perspectives. These authors proved  inequalities on $\mathrm{dim} \ M(L)$, involving invariants related  with the presentation of $L$. We know that $L \simeq F/R$  by the short exact sequence $0 \rightarrow R   \rightarrowtail F \twoheadrightarrow L \rightarrow 0$, where $F$ is a \textit{free Lie algebra} $F$ on $n$ generators,  $R$ an ideal of $F$  and Witt's Formula, which can be found in \cite{sal1, bosko2,  anc3, rot}, shows
\begin{equation}\label{moebius}
\mathrm{dim} \ F^d/F^{d+1} = \frac{1}{d} \sum_{r | d}\mu(r) \ n^{\frac{d}{r}} \equiv l_n(d),
\end{equation}
where
\[\mu(r) =\left\{\begin{array}{lcl}  1, & \mathrm{if} \ r=1,\\
0, & \mathrm{if} \ r \ \mathrm{is} \  \mathrm{divisible} \ \mathrm{by} \ \mathrm{a}  \ \mathrm{square},\\
(-1)^s, & \mathrm{if} \ r=p_1 \ldots p_s \  \mathrm{for} \ \mathrm{distinct} \ \mathrm{primes}  \ p_1, \ldots, p_s\end{array}\right.\]
is the celebrated \textit{M\"obius function}.   Now \cite[Theorem 2.5]{sal1} and similar results of \cite{sal1, bosko1, bosko2, sal2, sal3} provide inequalities of the same nature of  \eqref{yank} and \eqref{nice} but based on \eqref{moebius} and the main problem is to give an explicit expression for $l_n(d)$.
For instance,  if $c$ denotes the class of nilpotence of $L$, then \cite[Theorem 4.1]{bosko2} shows
\begin{equation}\label{bad}
\mathrm{dim} \ M(L) \le \sum_{j=1}^c l_n(j+1) = \sum_{j=1}^c  \ \left( \frac{1}{j+1}\sum_{i | j+1} \mu(i) \ n^{\frac{j+1}{i}} \right)
\end{equation}
and \cite[Examples 4.3, 4.4]{bosko2} provide explicit values for $\mu(m)$ in order to evaluate numerically \eqref{bad} and then to compare with \eqref{batten}. It is in fact hard to describe the behaviour of the M\"obius function from a general point of view and so \eqref{moebius} is not very helpful in the practice, when we do not evaluate the coefficients $\mu(m)$. We note briefly that  \cite[Theorem 3.2.5]{karp} is the corresponding version of  \eqref{bad} for groups and  the same problems happen also here.

Now we may understand the importance of being as much concrete as possible in the study of the upper bounds for $\mathrm{dim} \ M(L)$. We should also recall that  $A(n)$ denotes the abelian Lie algebra of dimension $n$ and the main results of \cite{sal1, es1, bosko1, bosko2, sal2, sal3,  es2, n1, n2, n3, n6} illustrate that many inequalities on $\mathrm{dim} \ M(L)$  become equalities if and only if $L$ splits in the sums of $A(n)$ and of a \textit{Heisenberg algebra} $H(m)$ (here $m \ge 1$ is a given integer).
To convenience of the reader, we recall that a finite dimensional
Lie algebra $L$ is called $Heisenberg$ provided that $L^2 = Z(L)$
and $\mathrm{dim} \ L^2 = 1$. Such algebras are odd dimensional with
basis $v_1, \ldots , v_{2m}, v$ and the only non-zero multiplication
between basis elements is $[v_{2i-1}, v_{2i}] = - [v_{2i},
v_{2i-1}]= v$ for $i = 1, \ldots ,m$.  Unfortunately, theorems of splitting of the aforementioned papers hold  only for small values of $m$ and $n$ (see for instance \cite[Theorem 3.1]{n1} or \cite[Theorems 2.2, 3.1, 3.5, 3.6, 4.2]{n2})  and we are very far from controlling the general cases. In fact,  Chao \cite{chao}  and Seeley \cite{see} proved that there exist uncountably many non--isomorphic nilpotent Lie algebras of finite dimension, beginning already from dimension 10 and this illustrates the complexity of the problem. At this point, we may state the first main result.

\begin{thm}\label{t:1} Let $L$ be a nilpotent Lie algebra of $\mathrm{dim}  \ L=n$, $\mathrm{dim} \ L^2=m$ and $\mathrm{dim} \ Z(L)=d$. If $L$ is nonabelian, then \eqref{nice} is better than \eqref{salemkar} for all $n\ge3$, $d \ge 1$  and $m \le \Big\lfloor \frac{n-2}{d+1}\Big\rfloor$.
\end{thm}



We may be more specific in the nilpotent case and use  certain exact sequences due to Ganea and Stallings \cite[Theorem 2.5.6]{karp} which have been adapted recently in \cite{sal4,  ris1, ris2, sal6, n7} to Lie algebras. Some notions of homological algebra should be recalled, in order to formulate the next result. The \textit{Schur multiplier of the pair} $(L,N)$, where $L$ is a Lie algebra with ideal $N$, is the abelian Lie algebra $M(L,N)$ which appears in the following  natural exact sequence of Mayer--Vietoris type
\begin{equation}\label{pairs}H_3(L) \longrightarrow H_3(L/N) \longrightarrow M(L,N)  \longrightarrow M(L) \longrightarrow M(L/N) \longrightarrow \end{equation}
\[ \longrightarrow \frac{L}{[L,N]} \longrightarrow \frac{L}{L^2} \ \longrightarrow \ \frac{L}{L^2 + N} \ \longrightarrow 0 \]
where the third homology (with integral coefficients) $H_3(L)$ of $L$ and $H_3(L/N)$ of $L/N$ are involved. We also recall that $\Phi(L)$ denotes the \textit{Frattini} subalgebra of $L$, that is, the intersection of all maximal subalgebras of $L$ (see \cite{marshall,n7}). It is easy to see that $\Phi(L)$ is an ideal of $L$, when $L$ is finite dimensional and nilpotent.

\begin{thm}\label{t:4} Let $L$ be a nilpotent Lie algebra of $\mathrm{dim} \ L=n$ and $N$ an ideal of $L$ of $\mathrm{dim} \ L/N =u$. Then \[ \mathrm{dim} \ M(L,N) \ + \ \mathrm{dim} \ [L,N] \le \frac{1}{2}n(2u+n-1).\] Furthermore, if $\mathrm{dim} \ L/(N+\Phi(L))=s$ and $\mathrm{dim} \ N/N \cap \Phi(L)=t$, then
\[\frac{1}{2}t(2s+t-1) \le \mathrm{dim} \ M(L,N) \ + \ \mathrm{dim} \ [L,N].\]
\end{thm}

When $N=L$ in Theorem \ref{t:4}, we get $u=0$ and find again \eqref{batten} so that Theorem \ref{t:4} is a generalization of \eqref{batten}. On the other hand, Theorem \ref{t:4} improves most of the bounds in \cite[Theorem B]{sal2}, where $L$ is assumed to be factorized.

The last  main theorem describes a condition of complementation. We may introduce $M(L,N)$ functorially from a wider perspective (originally, this is due to Ellis in \cite{ellis} for groups).  Let $B(L)$ be a classifying space such that
\begin{itemize}
\item[(i)] The topological space $B(L)$ is a connected CW--complex;
\item[(ii)] There exists a functor $\pi_n$ from the category of topological spaces to that of  Lie algebras such that $\pi_1(B(L)) \simeq L$;
\item[(iii)] The Lie algebras $\pi_n(B(L))$ are trivial for all $n \ge 2$.
\end{itemize}
Since the homology Lie algebras $H_n(B(L))$ (with integral coefficients) depend only on $L$, we have $H_n (L) = H_n (B(L))$ for all $n \ge 0$.
For each ideal $I$ of $L$ we may construct functorially a space $B(L,I)$
as follows. The quotient homomorphism $L \twoheadrightarrow  L/I$ induces a map $f : B(L) \rightarrow
B(L/I)$. Let $M(f)$ denote the mapping cylinder of this map. Note that $B(L)$
is a subspace of $M(f)$, and that $M(f)$ is homotopy equivalent to $B(L/I)$. We
take $B(L,I)$ to be mapping cone of the cofibration $B(L) \rightarrow M(f)$.
The cofibration sequence \[B(L) \longrightarrow M (f) \longrightarrow B(L,I)\] yields a natural long
exact homology sequence of Mayer--Vietoris
\[\ldots \longrightarrow  H_{n+1}(L/I) \longrightarrow H_{n+1} B(L,I) \longrightarrow H_n(L) \longrightarrow H_n(L/I) \longrightarrow \ldots
 \ \ \forall n \ge 0.\] It is straightforward to see that
\[H_1(B(L,I)) = 0 \ \mathrm{and} \ H_2(B(L,I)) \simeq I/[L,I] \ \mathrm{and} \ M(L,I)=H_3(B(L,I)).\]
These complications of homological nature have a positive consequence: we may treat the topic with more generality. By a \textit{triple} we mean a Lie algebra $L$ with two ideals $I$ and $J$ and by \textit{homomorphism of triples} $(L,I,J) \rightarrow (L',I',J')$ we mean a  homomorphism of Lie algebras $L \rightarrow L'$ that sends $I$ into $I'$ and $J$ into $J'$. The \textit{Schur multiplier of the triple}
$(L,I,J)$ will be the functorial abelian Lie algebra $M(L,I,J)$ defined by the
 natural exact sequence
\begin{equation}\label{triples}H_3(L,I) \longrightarrow H_3\left(\frac{L}{I}, \frac{I+J}{J}\right) \longrightarrow M(L,I,J) \longrightarrow M(L,J)
\longrightarrow M \left(\frac{L}{I}, \frac{I+J}{I}\right) \longrightarrow \end{equation}
\[ \longrightarrow \frac{I \cap J}{[L, I \cap J]+[I,J]}
\longrightarrow \frac{J}{[L,J]}  \longrightarrow \frac{I+J}{I+[L,J]} \longrightarrow 0,\]
where $H_3(L,I)=H_4(B(L,I))$ and $M(L,I,J)=H_4(B(L,I,J))$ is defined in terms of the mapping cone $B(L,I,J)$ of the canonical cofibration $B(L,I) \twoheadrightarrow B(L/J,I+J/I)$. Our last result is the following.

\begin{thm}\label{t:5} Let $L$ be a finite dimensional Lie algebra with two ideals $I$ and $J$ of $L$ such that $L=I+J$ and $I \cap J=0$. Then
\[\mathrm{dim} \ M(L,I,J)=\mathrm{dim} \ M(L,J) \ - \ \mathrm{dim} \ M(J). \] Moreover, if  $K \subseteq J \cap Z(L)$, then
\[\mathrm{dim} \ M(L,I,J)  +  \mathrm{dim} \ M(J) +  \mathrm{dim} \ K \cap [L,J]\]
\[ \le \mathrm{dim} \  M (L/K,J/K)  +  \mathrm{dim} \  M(K)  +  \mathrm{dim} \  \frac{L}{L^2+K} \cdot \mathrm{dim} \ K.\]
\end{thm}

\section{Proofs of the  results}

The following property deals with the low dimensional homology of sums and is a crucial instrument in the proofs of our main results.

\begin{lem}[See \cite{rot}, Theorem 11.31, K\"unneth Formula] \label{l:1}  Two finite dimensional Lie algebras $H$ and $K$ satisfy  the condition \[M(H\oplus K)=M(H) \oplus M(K) \oplus (H/H^2 \otimes K/K^2). \] In particular, \[\mathrm{dim} \ M(H \oplus K)= \mathrm{dim} \ M(H) \  + \ \mathrm{dim} \ M(K) \ + \ \mathrm{dim} \ H/H^2 \otimes K/K^2.\]
\end{lem}

The dimension of the Schur multiplier of abelian Lie algebras is a classic.

\begin{lem}[See \cite{es1}, Lemma 3]\label{l:3} $L \simeq A(n)$  if and only if $\mathrm{dim} \ M(L) =\frac{1}{2}n(n-1)$.
\end{lem}

Now we may specify \eqref{salemkar}.

\begin{lem}\label{l:4} If  a nilpotent Lie algebra $L$ of $\mathrm{dim} \ L=n$ has $\mathrm{dim} \ L^2=m$ and $\mathrm{dim} \ Z(L)=d$, then \eqref{salemkar} becomes
\[\mathrm{dim} \ M(L) \le \frac{1}{2}(n-m)(n-m-1)+m(n-d-1).\]
\end{lem}
\begin{proof} This is an application of Lemma \ref{l:3}, noting that $\mathrm{dim} \ L/L^2 =\mathrm{dim} \ L - \mathrm{dim} \ L^2= \mathrm{dim} \ A(n-m)=n-m$ and $\mathrm{dim} \ L/Z(L)= \mathrm{dim} \ L - \mathrm{dim} \ Z(L)=m-d$.
\end{proof}

\begin{proof}[Proof of Theorem \ref{t:1}] From Lemma \ref{l:4}, \eqref{salemkar} becomes
\[\mathrm{dim} \ M(L) \le \frac{1}{2}(n-m)(n-m-1)+m(n-d-1)\]
\[=\frac{1}{2}(n^2-nm-n-nm+m^2+m)+mn-dm-m\]
\[=\frac{1}{2}(n^2+m^2+m-n)-dm-m\]
\[=\frac{1}{2}(n^2+m^2)+\frac{1}{2}m-m-dm-\frac{1}{2}n\]
\[=\frac{1}{2}(n^2+m^2)-\frac{1}{2}m-dm-\frac{1}{2}n\]
\[=\frac{1}{2}(n^2+m^2)-\left(d+\frac{1}{2}\right)m-\frac{1}{2}n.\]
On the other hand, \eqref{nice} becomes
\[\mathrm{dim} \ M(L) \le \frac{1}{2}(n+m-2)(n-m-1)+1\]
\[=\frac{1}{2}(n^2-nm-n+nm-m^2-m-2n+2m+2)+1\]
\[=\frac{1}{2}(n^2-m^2)+\frac{1}{2}m-\frac{3}{2}n+2.\]
Of course, the first terms satisfy $\frac{1}{2}(n^2-m^2) \le \frac{1}{2}(n^2+m^2)$ for all $m,n \ge1$, but the remaining terms satisfy
\[\frac{1}{2}m-\frac{3}{2}n+2 \le -\left(d+\frac{1}{2}\right)m-\frac{1}{2}n \Leftrightarrow 0 \le  -(d+1) m + n - 2\]
\[ \Leftrightarrow 0 \ge (d+1)m-n+2 \Leftrightarrow  m \le  \Big\lfloor \frac{n-2}{d+1}\Big\rfloor. \]
It follows that \eqref{nice} is better than \eqref{salemkar} for these values of $m$.
\end{proof}

In order to prove Theorem \ref{t:4},  we should note that the Schur multiplier of a pair $(L,N)$ induce the following exact sequence
 \begin{equation}\label{schurpairs}   \longrightarrow M(L,C) \longrightarrow M(L,N) \longrightarrow  M(L/C, N/C) \longrightarrow  0 ,\end{equation}
where $C \subseteq Z(N)$.
Moreover, it is not hard to check that
the following exact sequence is induced by natural epimorphisms
\begin{equation}\label{natural} C \otimes \frac{L}{L^2+C}  \longrightarrow M(L,C). \end{equation}

Now we may prove Theorem \ref{t:4}.

\begin{proof}[Proof of Theorem \ref{t:4}]
We begin to prove the lower bound. We claim that
\[\mathrm{dim} \ M\left(\frac{L}{\Phi(L)}, \frac{N}{N \cap \Phi(L)} \right) \le  \mathrm{dim} \ M(L,N) + \mathrm{dim} \ [L,N]. \leqno(\dag)\]
Note from  \cite[Corollary 2, p.420]{marshall} that $\Phi(L)=L^2$ is always true for nilpotent Lie algebras. Then $L/\Phi(L)$ and $N/N \cap \Phi(L) \simeq N + \Phi(L)/\Phi(L) \subseteq L/\Phi(L)$ are abelian.  In our situation,
\[  M\left(\frac{L}{\Phi(L)}, \frac{N}{N \cap \Phi(L)} \right) \simeq  \frac{L}{\Phi(L)} \wedge \frac{N}{N \cap \Phi(L)}  \simeq \frac{L}{L^2} \wedge \frac{N}{N \cap L^2}. \]
Now we should recall from \cite{n2,n3,n4,n7} that  this is a classical situation in which it is possible to consider certain compatible actions by conjugation of $L$ over $N$ (and viceversa) which allows us to construct  the nonabelian tensor product $L \otimes N$ of $L$ and $N$. This construction has several properties, useful to our scopes. For instance, one can see that  $M(L,N) \simeq L \wedge N = L \otimes N/ L \square N,$ where
$L \square N =\langle x \otimes x  \ | \ x \in L \cap N \rangle $, and that the map
\[\kappa' : x \wedge y \in L \wedge N \longmapsto \kappa'(x \wedge y)=[x,y]\in [L,N]\]
is an epimorphism of Lie algebras with $\ker \kappa ' = M(L,N)$  such that
\[\mathrm{dim} \ L \wedge N = \mathrm{dim} \ M(L,N) + \mathrm{dim} \ [L,N].\]
On the other hand,  \[x \wedge y \in L \wedge N  \longmapsto x+L^2 \wedge y+(N\cap L^2) \in \frac{L}{L^2} \wedge \frac{N}{N \cap L^2}\] is also an epimorphism of Lie algebras and it  implies \[\mathrm{dim} \ L \wedge N \ge \mathrm{dim} \  \frac{L}{L^2} \wedge \frac{N}{N \cap L^2} \] so that
\[\mathrm{dim} \   M\left(\frac{L}{\Phi(L)}, \frac{N}{N \cap \Phi(L)} \right) = \mathrm{dim} \  \frac{L}{L^2} \wedge \frac{N}{N \cap L^2} \le \mathrm{dim} \ L \wedge N.\] The claim $(\dag)$  follows. Consequently, it will be enough to prove
\[\mathrm{dim} \ M\left(\frac{L}{\Phi(L)}, \frac{N}{N \cap \Phi(L)}\right) \le \frac{1}{2}t(2s+t-1) \leqno{(\dag \dag)} \]
 in order to conclude
\[\frac{1}{2}t(2s+t-1) \le \mathrm{dim} \ M(L,N) \ + \ \mathrm{dim} \ [L,N].\]
Since $N/N \cap \Phi(L) \simeq A(s)$ is a direct factor of the abelian Lie algebra $L/\Phi(L) \simeq A(s+t) \simeq A(s) \oplus A(t)$,
Lemma \ref{l:3} implies
\[\mathrm{dim} \ M\left(\frac{L}{\Phi(L)}\right)=\frac{1}{2}(s+t)(s+t-1) \ \mathrm{and} \ \mathrm{dim} \ M\left(\frac{L}{N+ \Phi(L)}\right)=\frac{1}{2}s(s-1).\]
On the other hand, we have (see for instance \cite[p.174]{ris1}) that
\[\mathrm{dim} \ M\left(\frac{L}{\Phi(L)}, \frac{N}{N \cap \Phi(L)}\right) =\mathrm{dim} \ M\left(\frac{L}{\Phi(L)}\right)- \mathrm{dim} \ M\left(\frac{L}{N+\Phi(L)}\right)\]
\[=\frac{1}{2}(s+t)(s+t-1) - \frac{1}{2}s(s-1) = \frac{1}{2}(s+t)(s+t-1)-\frac{1}{2}s(s-1)=\frac{1}{2}t(2s+t-1),\]
and so $(\dag \dag)$ is proved.

Now we prove the upper bound
\[\mathrm{dim} \ M(L,N) \le \frac{1}{2}n(2u+n-1).\]
Proceed by induction on $n$. Of course, the above inequality is true  when $n = 0$.
Suppose that it holds whenever $N$ is of dimension strictly less than $n$, and suppose
that $\mathrm{dim} \ N  = n$ . Let $C$ be a monodimensional Lie algebra contained in the center of $N$.
We are in the situation described by \eqref{schurpairs} and \eqref{natural}. Then
\[\mathrm{dim} \ M(L,N) \le
\mathrm{dim} \ M(L,C)+ \mathrm{dim} \ M \left(\frac{L}{C}, \frac{N}{C}\right)\]
\[ \le \mathrm{dim} \ C \otimes \frac{L/C}{(L/C)^2} \ + \ \mathrm{dim} \ M \left(\frac{L}{C}, \frac{N}{C}\right) \le \mathrm{dim} \ \frac{L/C}{(L/C)^2} +  \mathrm{dim} \ M \left(\frac{L}{C}, \frac{N}{C} \right)\]
\[=\mathrm{dim} \ \frac{L}{C} - \mathrm{dim} \ \left(\frac{L}{C}\right)^2 +  \mathrm{dim} \ M \left(\frac{L}{C}, \frac{N}{C} \right)
\le (u+n-1) +\frac{1}{2}(n-1) + \frac{1}{2}(2u+n-2)\]
\[=\frac{1}{2}n(2u+n-1),\]
as wished.
\end{proof}

From \cite[Equation 2.2]{bosko2} and  \cite{sal4,sal1,sal2,sal3},  we have a good description of the 2--dimensional homology over quotients and subalgebras. In fact it is possible to overlap the celebrated sequences of Ganea and Stallings in \cite[Theorem 2.5.6]{karp}, studied for groups almost 30 years ago, in the context of Lie algebras  (see  \cite{sal4, ris1, ris2, n7} for more details). We should also recall that a finite dimensional  Lie algebra $L$ is \textit{capable} if $L \simeq E/Z(E)$ for a suitable finite dimensional Lie algebra $E$ (see \cite{sal4, sal3, n6, n7}). The smallest  central subalgebra of a finite dimensional Lie algebra $L$ whose factor is capable is the \textit{epicenter} of $L$ and is denoted by $Z^*(L)$ (see \cite{sal4, n7}). Notice that finite dimensional Lie algebras are capable if and only if they have trivial epicenter, so that this ideal gives a measure of how far we are from having a Lie algebra which may be expressed as a central quotient.

\begin{lem}[See \cite{sal4}, Proposition 4.1 and Theorem 4.4]\label{l:5} Let $L$ be a finite dimensional  Lie algebra and $Z$ a central ideal of $L$. Then  the following  sequences are exact:
\begin{itemize}
\item[(i)] $ Z \otimes L/L^2   \rightarrow M(L) {\overset{\beta}\longrightarrow} M(L/Z) {\overset{\gamma}\longrightarrow} L^2 \cap Z \rightarrow 0$,
\item[(ii)] $M(L) {\overset{\beta}\longrightarrow} M(L/Z) \rightarrow Z \rightarrow L/L^2 \rightarrow L/Z+L^2 \rightarrow 0$,
\end{itemize}
where $\beta$ and $\gamma$ are induced by natural embeddings. Moreover, the following conditions are equivalent:
\begin{itemize}
\item[(j)] $M(L) \simeq M(L/Z)/(L^2 \cap Z)$,
\item[(jj)] The  map  $ \beta: M(L) \rightarrow M(L/Z)$ is a monomorphism,
\item[(jjj)] $Z \subseteq Z^*(L)$.
\end{itemize}
\end{lem}

The next corollary is an immediate consequence of the previous lemma.

\begin{cor}[See \cite{sal4}, Corollaries 4.2 and 4.5]\label{c:1} With the notations of Lemma \ref{l:5},
\[ \mathrm{dim} \ M(L/Z)  \le \mathrm{dim} \ M(L) + \mathrm{dim} \ L^2 \cap Z \le \mathrm{dim} \ M(L/Z) + \mathrm{dim} \ Z \cdot \mathrm{dim} \ L/L^2\]
and, if $Z \subseteq Z^*(L)$, then \[ \mathrm{dim} \ M(L) + \mathrm{dim} \ L^2 \cap Z = \mathrm{dim} \ M(L/Z). \]
\end{cor}

 To convenience of the reader, given an ideal $I$ of a finite dimensional Lie algebra $L$ and a subalgebra $J$ of $L$, we recall that $I$ is said to be a complement of $J$ in $L$ if $L=I+J$ and $I \cap J =0$. The generalization of Corollary \ref{c:1} to the pair $(L,N)$ is illustrated by  Corollary \ref{c:2}, in which the assumption that the  ideal $N$ possesses a complement in $L$ implies (see for instance \cite[p.174]{ris1}) the isomorphism
 \begin{equation}\label{general}M(L) \simeq M(L,N) \oplus M(L/N).\end{equation}

\begin{cor}[See \cite{ris1}, Theorems 2.3 and 2.8]\label{c:2}
Let $N$ be a complement of a finite dimensional Lie algebra $L$ and $K$ be an ideal of $L$. If  $K \subseteq N \cap Z(L)$,  then
\[\mathrm{dim} \ M(L,N) + \mathrm{dim} \  K \cap [L,N]  \]
\[\le \mathrm{dim} \ M(L/K, N/K) + \mathrm{dim} \ M(K) + \mathrm{dim} \ \frac{L}{L^2+K} \cdot \mathrm{dim} \ K. \] Moreover, if $K \subseteq N \cap Z^*(L)$, then
\[\mathrm{dim} \ M(L,N) + \mathrm{dim} \  K \cap [L,N] = \mathrm{dim} \ M(L/K,N/K).\]
\end{cor}

The proof of Corollary \ref{c:2} is based on the exactness of the sequence \eqref{pairs}, when $N$ is a complement of a finite dimesional Lie algebra $L$ and we are going to generalize this strategy in our last result.

\begin{proof}[Proof of Theorem \ref{t:5}]
Since $I,J$ are ideals of $L$ and $J$ is a complement of $I$ in $L$, the exact sequence \eqref{triples}  induces the short exact sequence
\[\ 0  \rightarrow M(L,I,J) \rightarrowtail  M(L,J)  \twoheadrightarrow M(L/I,L/I) \rightarrow 0,\]
that is,  $M(L,J)$ splits over $M(L,I,J)$ by $M(L/I, L/I)=M(L/I)=M(J)$. Then
\[\mathrm{dim} \ M(L,J) = \mathrm{dim} \ M(L,I,J) + \mathrm{dim} \ M(J).\]
We may apply Corollary \ref{c:2} to the term $\mathrm{dim} \ M(L,J)$, getting
\[\mathrm{dim} \ M(L,I,J) = \mathrm{dim} \ M(L,J) - \mathrm{dim} \ M(J)\]
\[\le \Big(\mathrm{dim} \  M (L/K,J/K) \ + \ \mathrm{dim} \  M(K) \ + \ \mathrm{dim} \  \frac{L}{L^2+K} \otimes K - \mathrm{dim} \ K \cap [L,J]
\Big) - \mathrm{dim} \ M(J)\]
\[= \mathrm{dim} \  M (L/K,J/K) \ - \mathrm{dim} \ M(J) \ + \ \mathrm{dim} \  M(K) \]\[ + \ \mathrm{dim} \  \frac{L}{L^2+K} \otimes K- \mathrm{dim} \ K \cap [L,J],\]
as claimed.
\end{proof}

A special case is the following.

\begin{cor} With the notations of Theorem \ref{t:5}, if $K \subseteq J \cap Z^*(L)$, then
\[\mathrm{dim} \ M(L,I,J)=  \mathrm{dim} \  M (L/K,J/K) - \mathrm{dim} \ K \cap [L,J]- \mathrm{dim} \ M(J).\]
\end{cor}

\end{document}